\documentclass[12pt]{amsart}
\usepackage{amsmath,amsfonts,amssymb,amsthm}
\usepackage{mathrsfs}

\newcommand{\Z}{{\mathbb Z}}
\newcommand{\R}{{\mathbb R}}

\newcommand{\N}{{\mathbb N}}
\renewcommand{\le}{\leqslant}
\renewcommand{\ge}{\geqslant}
\newcommand{\la}{\langle}
\newcommand{\ra}{\rangle}
\newcommand{\wh}{\widehat}

\newtheorem{theorem}{Theorem}

\newtheorem{corollary}{Corollary}
\begin{document}

\title{On the bounds of coefficients of Daubechies orthonormal wavelets}

\author{Susanna Spektor}
\address{Department of Mathematical and Statistical Sciences, PSB,
Sheridan College Institute of Technology,  4180 Duke of York Blvd
Mississauga, Ontario L5B 0G5 }
\email{spektor.susanna@sheridancollege.ca}



%


\date{}

\maketitle

\thispagestyle{empty}

\begin{abstract} This article is a continuation of the studies published in \cite{Babenko.Spektor:2007}. In the present work we provide the bounds for Daubechies orthonormal wavelet coefficients  for  function spaces $\mathcal{A}_k^p:=\{f:
\|(i \omega)^k\hat{f}(\omega)\|_p< \infty\}$, $k\in\N\cup\{0\}$, $p\in(1,\infty)$.
\end{abstract}


\section{Introduction and motivations}

A function $\psi$ is called a {{wavelet}} if there exists a dual function $\widetilde{\psi}$, such that any function $f \in L_2(\R)$ can be expressed in the form
\[
f(t)=\sum_{j \in \Z}\sum_{\nu \in \Z}\la f,\widetilde{\psi}_{j, \nu}\ra \psi_{j, \nu}(t).
\]

The development of wavelets goes back to A.~Haar's work in early 20-th century and to D.~Gabor's work (1946), who constructed functions similar to wavelets. Notable contributions to wavelet theory can be attributed to G.~Zweig's discovery of the continuous wavelet transform in 1975; D.~Goupilland, A.~Grossmann and J.~Morlet's formulation of the cosine wavelet transform (CWT) in 1982; J.~Str$\ddot{\textmd{o}}$mberg's work on discrete wavelets (1983);   I.~Daubechies' orthogonal wavelets with compact support (1988); Y.~Meyer's  orthonormal basis of wavelets (1989); S.~Mallat's multiresolution framework (1989);  and many others.

Wavelets are used in signal analysis, molecular dynamics, density-matrix localisation, optics, quantum mechanics, image processing, DNA analysis, speech recognition, to name few. Wavelets have such a wide variety of applications mainly because of their ability to encode a signal using only a few of the larger coefficients. The numbers of large coefficients depends on
\begin{itemize}
  \item [-] the size of the support of the signal: the shorter support the better;
  \item [-] the number of vanishing moments: the more vanishing moments a wavelet has, the more it oscillates. (The number of vanishing moments  determines what the wavelet does not see).
  \item [-] regularity (smoothness) of the signal: the number of continuous derivatives.
\end{itemize}

In general the Daubechies wavelets are chosen to have highest number $m$ of vanishing moments for a given support width $2m-1$.
Let $m \in \N$. The trigonometric polynomials
$$
H_m(\omega)=2^{-1/2}\sum_{\ell=0}^{2m-1}h_m(\ell)e^{i\ell \omega}, \quad h_m(\ell) \in \R,
$$
which satisfy the equalities
$$
|H_m(\omega)|^2=\left(\cos^2 \frac{\omega}{2}\right)^mP_{m-1}\left(\sin^2 \frac{\omega}{2}\right),
$$
where
$$
P_{m-1}(x)=\sum_{k=0}^{m-1}{m-1+k \choose k}x^k,
$$
are called Daubechies filters (see e.g. \cite{Daub:book}).

The Fourier transform, $\hat{f}$, of function $f \in L_1(\R)$ is defined to be $\displaystyle{\hat{f}(\omega)=\frac{1}{\sqrt{2 \pi}}\int_{\R}f(x)e^{-i\omega x}\, dx}$.

A function $\varphi_{m}^D$ whose Fourier transform has the form
$$
(\varphi_m^D)^{\wedge}(\omega)=\frac{1}{\sqrt{2 \pi}}\prod_{\ell=1}^{\infty}H_m(\omega 2^{-\ell})
$$
is the orthogonal scaling function. A function whose Fourier transform has the form
$$
(\psi_m^{D})^{\wedge}(\omega)=e^{-\frac{i \omega}{2}}\overline{H_m\left(\frac{\omega}{2}+\pi\right)}(\varphi_m^D)^{\wedge}\left(\frac{\omega}{2}\right)
$$
is called an orthogonal Daubechies wavelet $\psi_m^D$

Note (see e.g. \cite{SN}),
\begin{align*}
\left|(\psi_m^{D})^{\wedge}(\omega)\right|^2&=\left|H_m\left(\frac{\omega}{2}+\pi\right)\right|^2\left|(\varphi_m^D)^{\wedge}\left(\frac{\omega}{2}\right)\right|^2\\
&=
\frac{1}{2 \pi}\left|H_m\left(\frac{\omega}{2}+\pi\right)\right|^2\prod_{\ell=1}^{\infty}\left|H_m\left(2^{-\ell-1}{\omega}\right)\right|^2
\end{align*}
and
$$
\left|H_m\left({\omega}\right)\right|^2=1-\frac{\Gamma(m+1/2)}{\sqrt{\pi}\Gamma(m)}\int_0^{t}\sin^{2m-1}\omega\, d \omega.
$$

Many wavelet applications, for example, image/signal compression, denoising, inpainting, compressive sensing, and so on,
are based on investigation of the wavelet coefficients $\la f, \psi_{j,\nu}\ra$ for
$j,\nu\in\Z$, where $\la f,g\ra:=\int_\R f(x)\overline{g(x)}dx$ and
$\varphi_{j,\nu}:=2^{j/2}\varphi(2^j\cdot-\nu),\psi_{j,\nu}:=2^{j/2}\psi(2^j\cdot-\nu)$.
 The magnitude of the wavelet coefficients depends on both the
smoothness of the function $f$ and the wavelet $\psi$. In this
paper, we shall investigate the quantity
\begin{equation}\label{def:Ckp}
C_{k,p}(\psi)=\sup_{f\in \mathcal{A}_k^{p'}}\frac{|\la
f,\psi\ra|}{\|\hat\psi\|_p},
\end{equation}
where $1< p,p'< \infty$, $1/p'+1/p=1$, $k\in\N\cup\{0\}$, and
$\mathcal{A}_k^{p'}:=\{f\in L_{p'}(\R): \|(i
\omega)^k\hat{f}(\omega)\|_{p'}\le 1\}$.

Let us note here, that the quantity $C_{k,p}(\psi)$ in
\eqref{def:Ckp} is the best possible constant in the following
Bernstein type inequality
\begin{equation}\label{BernsteinTypeIneq}
|\la f, \psi_{j,\nu}\ra|\le
C_{k,p}(\psi)2^{-j(k+1/p-1/2)}\|\psi\|_p\|(i \omega)^k \wh
f(\omega)\|_{p'}
\end{equation}
Such type of inequalities plays an important role in wavelet
algorithms for the numerical solution of integral equations (see e.g.
\cite{Beylkin.Coifman.Rokhlin:1991}) where wavelet
coefficients arise by applying an integral operator to a wavelet and
bound of the type \eqref{BernsteinTypeIneq} gives  priori
information on the size of the wavelet coefficients. We refer to \cite{Babenko.Spektor:2007} for the more explanations about relation of $C_{k,p}(\psi)$ to the estimates of the expansion coefficients for the Daubechie's wavelets.

Note here,  that
\begin{equation}\label{def:Ckp2}
C_{k,p}(\psi)=\sup_{f\in \mathcal{A}_k^{p'}}\frac{|\la
f,\psi\ra|}{\|\hat\psi\|_p}=\sup_{f\in \mathcal{A}_k^{p'}}\frac{|\la
\hat{f},\hat{\psi}\ra|}{\|\hat\psi\|_p}=\frac{\|\wh{_k\psi}\|_{p}}{\|\hat{\psi}\|_p},
\end{equation}
where for a function $f\in L_1(\R)$, $_kf$ is defined to be the
function such that
\begin{equation}\label{def:kf}
\wh{_kf}(\omega)=(i\omega)^{-k}\hat f(\omega),\quad \omega\in\R.
\end{equation}


\section{Main Results}

\begin{theorem}\label{TH1}
Let $k\geq 0$ be a nonnegative integer. Let $\psi^D_m$ be the Daubechies orthonormal  wavelet of
order $m$, $m>k$. Then, for $p \in (1, \infty)$, $\widetilde{C}, c>0$, $\varepsilon \in (0, \pi)$,
\begin{equation}\label{eq1}
B\leq \|\wh{_k(\psi^D_m)}\|_p\leq A,
\end{equation}
where
\begin{align*}
A=A(p,k,m, \widetilde{C}, c)&=\frac{(2 \pi)^{1/p-1/2}}{\pi^k}\left(\frac{1-2^{1-pk}}{pk-1}\right)^{1/p}\\
&+\Big[ 2^{1-p\left(2m+\frac 12\right)}\pi^{p(m-k-\frac 12)+1}\left(\frac{(2m)!}{m! (m-1)!}\right)^{p/2}\\
&+\frac{1}{(2 \pi)^{cp \log m-2}}+\frac{1}{2^{p/2-1}\pi^{p(k+\frac 12)-1}}\Big]^{\frac 1p}
\end{align*}
and
\begin{align*}
B=B(p,k,m, \widetilde{C}, c)&=\frac{(2 \pi)^{1/p-1/2}}{\pi^k}\left(\frac{1-2^{1-pk}}{pk-1}\right)^{1/p}\\
&-\Big[ 2^{1-p\left(2m+\frac 12\right)}\varepsilon^{p(m-k-\frac 12)+1}\left(\frac{(2m)!}{m! (m-1)!}\right)^{p/2}\\
&+\frac{1}{(2 \pi)^{cp \log m-2}}+\frac{1}{2^{p/2-1}\pi^{p(k+\frac 12)-1}}\Big]^{\frac 1p}
\end{align*}
\end{theorem}

\begin{proof}
The first part f the proof is going along the lines of the asymptotic version of this theorem in \cite{Babenko.Spektor:2007}.

Let $\displaystyle{\hat{\Psi}=\frac{1}{\sqrt{2 \pi}}\left(\chi_{[-2\pi, -\pi]}+\chi_{[\pi, 2\pi]}\right)}$, where $\chi_I$ is the characteristic function of the interval $I$.

Using Minkowski inequality, we obtain,

\begin{align*}
I:= \|\wh{_k(\psi^D_m)}\|_p&=\left(\int_{\R}|(i \omega)^{-k}|^p |\wh{\psi^D_m}|^p \, d \omega\right)^{1/p}\\
&=\left(\int_{\R}|\omega|^{-pk} |\hat{\Psi}(\omega)-(\hat{\Psi}(\omega)-\wh{\psi^D_m})|^p \, d \omega\right)^{1/p}\\
&\leq \left(\int_{\R}|\omega|^{-pk} |\hat{\Psi}(\omega)|^p \, d \omega\right)^{1/p}+\left(\int_{\R}|\omega|^{-pk} |\hat{\Psi}(\omega)-\wh{\psi^D_m}|^p \, d \omega\right)^{1/p}\\
&:=I_1+I_2.
\end{align*}
On the other hand,
\begin{align*}
I:= \|\wh{_k(\psi^D_m)}\|_p&=\left(\int_{\R}|\omega|^{-pk} |\hat{\Psi}(\omega)-(\hat{\Psi}(\omega)-\wh{\psi^D_m})|^p \, d \omega\right)^{1/p}\\
&\geq \left(\int_{\R}|\omega|^{-pk} |\hat{\Psi}(\omega)|^p \, d \omega\right)^{1/p}-\left(\int_{\R}|\omega|^{-pk} |\hat{\Psi}(\omega)-\wh{\psi^D_m}|^p \, d \omega\right)^{1/p}\\
&:=I_1-I_2.
\end{align*}

Consider separately
\begin{align*}
I_1&= \left(\int_{\R}|\omega|^{-pk} |\hat{\Psi}(\omega)|^p \, d \omega\right)^{1/p}\\
&=\frac{1}{\sqrt{2 \pi}}\left(\int_{-2 \pi}^{-\pi} \omega^{-pk}\, d \omega+\int_{ \pi}^{2\pi} \omega^{-pk}\, d \omega\right)^{1/p}\\
&=\frac{(2 \pi)^{1/p-1/2}}{\pi^k}\left(\frac{1-2^{1-pk}}{pk-1}\right)^{1/p}.
\end{align*}

It is only left to find an upper bound for
\begin{align*}
I_2:= \left(\int_{\R}|\omega|^{-pk} |\hat{\Psi}(\omega)-\wh{\psi^D_m}|^p \, d \omega\right)^{1/p}.
\end{align*}

Fix $\varepsilon \in (0, \pi)$. Then,
\begin{align*}
I_2^p=\int_{|\omega|< \varepsilon}|\omega|^{-pk}|\wh{\psi_m^D}(\omega)-\wh{\Psi}(\omega)|^p\, d \omega+ \int_{|\omega|> \varepsilon}|\omega|^{-pk}|\wh{\psi_m^D}(\omega)-\wh{\Psi}(\omega)|^p\, d \omega :=I_{21}+I_{22}.
\end{align*}

Since $\wh{\Psi}(\omega)=0$ for all $\omega \in (-\pi, \pi)$, we have
\begin{align*}
I_{21}&:=\int_{|\omega|< \varepsilon}|\omega|^{-pk}|\wh{\psi_m^D}(\omega)-\wh{\Psi}(\omega)|^p\, d \omega\\
&= \int_{|\omega|< \varepsilon}|\omega|^{-pk}|\wh{\psi_m^D}(\omega)|^p\, d \omega\\
&\leq \left(\frac{1}{2 \pi}\right)^{p/2}\int_{\omega|< \varepsilon}|\omega|^{-pk}\left|H_m\left(\frac{\omega}{2}+\pi\right)\right|^p\, d \omega\\
&\leq \left(\frac{\Gamma(m+1/2)}{2\pi^{3/2}\Gamma(m)}\right)^{p/2}\int_{|\omega|< \varepsilon}|\omega|^{-pk}\left(\int_0^{|\omega|/2}\sin^{2m-1}t\, dt\right)^{p/2}\, d\omega.
\end{align*}
Using the fact that  $\displaystyle{\left(\frac{|\omega|}{2}\right)^{-1}\sin\left(\frac{|\omega|}{2}\right)\leq 1}$ and that $k<m$, we obtain
\begin{align*}
I_{21}&\leq 2^{-p(1/2+k)} \left(\frac{\Gamma(m+1/2)}{2\pi^{3/2}\Gamma(m)}\right)^{p/2} \int_{|\omega|< \varepsilon}\left(\frac{|\omega|}{2}\right)^{-p(k-1/2)}\left(\sin^{2m-1}\frac{|\omega|}{2}\right)^{p/2}\, d\omega\\
&\leq 2^{-p(1/2+k)} \left(\frac{\Gamma(m+1/2)}{2\pi^{3/2}\Gamma(m)}\right)^{p/2} \int_{|\omega|< \varepsilon}\sin^{p(m-k)}\frac{|\omega|}{2}\, d\omega\\
&\leq 2^{-p(1/2+k)} \left(\frac{\Gamma(m+1/2)}{2\pi^{3/2}\Gamma(m)}\right)^{p/2} \left(\varepsilon\right)^{p(m-k)+1}\\
&\leq 2^{-p(1/2+k)} \left(\frac{\Gamma(m+1/2)}{2\pi^{3/2}\Gamma(m)}\right)^{p/2}  \left(\pi\right)^{p(m-k)+1}\\
&=2^{-p(1/2+2m)+1} \pi^{p(m-k-1/2)+1}\left(\frac{(2m)!}{m! (m-1)!}\right)^{p/2}.
\end{align*}

In order to bound $\displaystyle{I_{21}:=\int_{|\omega|> \varepsilon}|\omega|^{-pk}|\wh{\psi_m^D}(\omega)-\wh{\Psi}(\omega)|^p\, d \omega}$, we notice that  $\wh{\Psi}(\omega)=0$ for all $|\omega| > 2\pi$. We break $I_{22}$ into three integrals which we have to bound from above.
\begin{align*}
I_{21}&:=\int_{-2\pi<\omega< \pi}|\omega|^{-pk}|\wh{\psi_m^D}(\omega)-\wh{\Psi}(\omega)|^p\, d \omega\\
&+\int_{\pi<\omega< 2\pi}|\omega|^{-pk}|\wh{\psi_m^D}(\omega)-\wh{\Psi}(\omega)|^p\, d \omega\\
&+\int_{|\omega|> \pi}|\omega|^{-pk}|\wh{\psi_m^D}(\omega)-\wh{\Psi}(\omega)|^p\, d \omega\\
&=I_{31}+I_{32}+I_{33}.
\end{align*}

In order to bound $I_{33}$ we use result from \cite{LMR} (Sec. 2.4.26), that there are constants $\widetilde{C}, c>0$, such that for any $\omega> 2 \pi$,
\[
|\wh{\psi_m^D}(\omega)|\leq \widetilde{C}|\omega|^{-c \log m}.
\]
Thus,
\[
I_{33}\leq (2 \pi)^{-cp \log m+2}\int_{|\omega|> 2 \pi}\omega^{-2}\, d \omega\leq (2 \pi)^{-cp \log m+2}.
\]
Integrals $I_{31}$ and $I_{32}$ have the same bound
\[
I_{32}\leq \pi^{-pk}\int_{\pi}^{2 \pi}\left|\wh{\psi_m^D}-\frac{1}{\sqrt{2 \pi}}\chi_{[\pi, 2 \pi]}\right|^{p}\, d \omega\leq 2^{-p/2} \pi^{1-p(k+1/2)}.
\]
Analogically, $\displaystyle{I_{31}\leq 2^{-p/2} \pi^{1-p(k+1/2)}.}$

We obtained the following bound for
\begin{align*}
I_2^p\leq 2^{1-p\left(2m+\frac 12\right)}\pi^{p(m-k-\frac 12)+1}&\left(\frac{(2m)!}{m! (m-1)!}\right)^{p/2}\\
&+\frac{1}{(2 \pi)^{cp \log m-2}}+\frac{1}{2^{p/2-1}\pi^{p(k+\frac 12)-1}}.
\end{align*}
\end{proof}

\begin{corollary}\label{Cor1}
 Let $\psi^D_m$ be the Daubechies orthonormal  wavelet of
order $m$. Then, for $p \in (1, \infty)$,
\begin{equation}\label{eq2}
E\leq \|(\psi^D_m)^{\wedge}\|_p\leq D,
\end{equation}
where
\begin{align*}
D=2(2 \pi)^{1/p-1/2}+2^{1/2-2m}\pi^{m+1/2}\left(\frac{(2m)!}{m!(m-1)!}\right)^{1/2}+(2 \pi)^{2-c \log m}
\end{align*}
and
\begin{align*}
E=(2 \pi)^{1/p-1/2}&-\Big[2^{-p(1/2+2m)+1}\pi^{p(m-1/2)+1}\left(\frac{(2m)!}{m!(m-1)!}\right)^{p/2}\\
&+(2 \pi)^{2-cp \log m}+(2 \pi)^{1-p/2}\Big]^{\frac 1p}.
\end{align*}
\end{corollary}

\begin{corollary}\label{Cor2}
 Let $\psi^D_m$ be the Daubechies orthonormal  wavelet of
order $m$. Then, for $p \in (1, \infty)$,
\begin{equation}\label{eq3}
\frac{B}{D}\leq C_{k,p}(\psi_m^D)=\frac{\|\wh{_k(\psi^D_m)}\|_p}{\|\psi_m^D\|_p}\leq \frac AE.
\end{equation}
\end{corollary}

\bigskip

We provide now a different method of bounding $\|(\psi_m^D)^{\wedge}\|_p$ for special case when $k=m$. This method gives us more accurate bounds.

\begin{theorem}\label{TH2}
 Let $\psi^D_m$ be the Daubechies orthonormal  wavelet of
order $m$. Then, for $p \in (1, \infty)$,  mp -- even,
\begin{equation}\label{eq4}
G\leq \|\wh{_m(\psi^D_m)}\|_p\leq F,
\end{equation}
where
\begin{align*}
F^p= \frac{2^{1-p/2}}{\pi^{p(m+1/2)-1}(mp-1)}+\frac{2^{1-p(2m-1)}}{\pi^{p/2-1}(mp-1)!}\sum_{i=0}^{\lfloor\frac{mp}{2}\rfloor}(-1)^i\frac{(mp)!}{(mp-i)!i!}(mp-2i)^{mp-1}
\end{align*}
and
\begin{align*}
G^p=\frac{(1-o(1))2^{1-2pm}(2m)!}{\pi^{p/2-1}m^{p/2}3^{mp}m!(m-1)!}.
\end{align*}
\end{theorem}

\begin{proof}The first part of the proof follows along the lines from the asymptotic variant of this theorem in \cite{SZ} (see Theorem 3 there). We provide detailes for the reader convenience.
\begin{align*}
\|\wh{_m\psi_m^D}\|_p^p &= \int_\R
|\omega|^{-mp}\left|\wh{\psi_m^D}(\omega)\right|^pd\omega\\
&=\int_{|\omega|\le\pi}|\omega|^{-mp}\left|\wh{\psi_m^D}(\omega)\right|^pd\omega+\int_{|\omega|> \pi}
|\omega|^{-mp}\left|\wh{\psi_m^D}(\omega)\right|^pd\omega =:I_1+I_2.
\end{align*}

We first estimate $I_2$. Since $|H_m(t)|\le1$,
\[
0\le I_2\le\frac{2}{(\sqrt{2\pi})^p}\int_{\pi}^\infty\omega^{-mp}d\omega
\le \frac{2}{(\sqrt{2\pi})^p}\cdot
\frac{1}{mp-1}\left(\frac{1}{\pi}\right)^{mp-1},\quad mp>1.
\]

Consider now
\[
\begin{aligned}
I_1&=\frac{1}{(\sqrt{2\pi})^p}\int_{|\omega|\le\pi}|\omega|^{-mp}
\Bigg[
\left|H_m\left(\omega/2+\pi\right)\right|^2\left|H_m\left(\omega/4\right)\right|^2\left|H_m\left(\omega/8\right)\right|^2
\\&\qquad\qquad\times \left. \prod_{\ell=1}^\infty\left|H_m\left(2^{-l-3}\omega\right)\right|^2
\right]^{p/2}d\omega
\\&
\ge\left(1-o(1)\right)\frac{1}{(\sqrt{2\pi})^p}\int_{|\omega|\le\pi}|\omega|^{-mp}\left|H_m\left(\omega/2+\pi\right)\right|^pd\omega.
\end{aligned}
\]
Obviously,
\[
I_1\le
\frac{1}{(\sqrt{2\pi})^p}\int_{|\omega|\le\pi}|\omega|^{-mp}\left|H_m(\omega/2+\pi)\right|^pd\omega.
\]
Now, we use the property of $H_m$ to deduce the bounds for
of
\[
I_{11}:=\int_{|\omega|\le\pi}|\omega|^{-mp}\left|H_m(\omega/2+\pi)\right|^pd\omega.
\]
Let ${u=\frac{\sin^2t}{\sin^2(\omega/2)}}$. We have, with $\displaystyle{c_m=\frac{\Gamma(m+1/2)}{\sqrt{\pi}\Gamma(m)}=\frac{(2m)!\sqrt{\pi}}{2^{2m}m!(m-1)!}}$
\[
\begin{aligned}
\left|H_m(\omega/2+\pi)\right|^2 &= c_m\int_0^{\omega/2}\sin^{2m-1}tdt
\\&=\frac{c_m}{2}\sin^{2m}(\omega/2)\int_0^1u^{m-1}(1-u\sin^2(\omega/2))^{-1/2}du.
\end{aligned}
\]
Since
\[
\frac1m=\int_0^1u^{m-1}du\le\int_0^1u^{m-1}(1-u\sin^2(\omega/2))^{-1/2}du\le\int_0^1u^{m-1}(1-u)^{-1/2}du=c_m^{-1}
\]
and
\[
\begin{aligned}
 I_{11}&= 2\int_0^{\pi}|\omega|^{-mp}\cdot
\left[\frac{c_m}{2}\sin^{2m}(\omega/2)\int_0^1u^{m-1}(1-u\sin^2(\omega/2))^{-1/2}du\right]^{p/2}d\omega,
\end{aligned}
\]
we obtain
\[
 \left(\frac{c_m}{2m}\right)^{p/2}\cdot
2^{-mp}\cdot\int_0^\pi\left(\frac{\sin(\omega/2)}{\omega/2}\right)^{mp}d\omega
\le \frac12 I_{11}\le
 \left(\frac12\right)^{-p/2}\cdot
2^{-mp}\cdot\int_0^\pi\left(\frac{\sin(\omega/2)}{\omega/2}\right)^{mp}d\omega.
\]
Making change of variables, $t=\omega/2$, we obtain:
\begin{align*}
\left(\frac{c_m}{m}\right)^{p/2}2^{2-p(m+1/2)}\int_0^{\pi}\left(\frac{\sin t}{t}\right)^{mp}\, dt \leq I_{11}\leq 2^{1-p(m-1/2)}\int_0^{\pi}\left(\frac{\sin t}{t}\right)^{mp}\, dt.
\end{align*}
In order to bound $\displaystyle{\int_0^{\pi/2}\left(\frac{\sin t}{t}\right)^{mp}\, dt}$ from above, we make restriction that $mp$ is even. We have
\[
\int_0^{\pi/2}\left(\frac{\sin t}{t}\right)^{mp}\, dt\leq \int_0^{\infty}\left(\frac{\sin t}{t}\right)^{mp}\, dt
=\frac{\pi}{2^{mp}(mp-1)!}\sum_{i=0}^{\lfloor\frac{mp}{2}\rfloor} (-1)^i{mp \choose i} (mp-2i)^{mp-1},
\]
where last formula can be found, for example, in \cite{B} (see example 22, p. 518 there).

Since for all $x \in \R$, one has $\displaystyle{\frac{\sin t}{t}\geq 1-\frac{t^2}{6}}$ (see e.g., \cite{H} or \cite{LD}), we get
\begin{align*}
\left(\frac{\sin t}{t}\right)^{mp}\, dt\geq \left(1-\frac{t^2}{6}\right)^{mp}\geq \left(1-\frac{(\pi/2)^2}{6}\right)^{mp}\geq\left(\frac 13\right)^{mp}\, dt=\frac{\pi\, 3^{-mp}}{2 }
\end{align*}
and
\begin{align*}
\int_0^{\pi/2}\left(\frac{\sin t}{t}\right)^{mp}\, dt\geq \int_0^{\pi/2}\left(\frac 13\right)^{mp}\, dt=\frac{\pi\, 3^{-mp}}{2}
\end{align*}

With that we have obtained bounds for $I_1$
\begin{align*}
\frac{(1-o(1))}{(2 \pi)^{p/2}}\frac{2^{1-p(2m+1/2)}}{3^{mp}m^{p/2}}\frac{(2m)!}{m!(m-1)!}&\leq I_1^p\\
& \leq \frac{2^{1-p(2m-1/2)}}{(2 \pi)^{p/2}(mp-1)!}\sum_{i=0}^{\lfloor\frac{mp}{2}\rfloor}(-1)^i\frac{(mp)!}{i!(mp-i)!}(mp-2i)^{mp-1},
\end{align*}
which is completes the proof.
\end{proof}

\begin{corollary}\label{Cor3}
 Let $\psi^D_m$ be the Daubechies orthonormal  wavelet of
order $m$. Then, for $p \in (1, \infty)$, mp -- even,
\begin{equation}\label{eq5}
\frac{G}{D}\leq C_{m,p}(\psi_m^D)=\frac{\|\wh{_m(\psi^D_m)}\|_p}{\|\psi_m^D\|_p}\leq \frac FE.
\end{equation}
\end{corollary}

\end{document}